\documentclass[12pt,a4paper]{amsart}
\allowdisplaybreaks[1]

\usepackage[top=30truemm,bottom=30truemm,left=25truemm,right=25truemm]{geometry}

\DeclareFontEncoding{OT2}{}{}
\DeclareFontSubstitution{OT2}{cmr}{m}{l}

\DeclareFontFamily{OT2}{cmr}{\hyphenchar\font45 }
\DeclareFontShape{OT2}{cmr}{m}{l}{%
<5><6><7><8><9>gen*wncyr%
<10><10.95><12><14.4><17.28><20.74><24.88>wncyr10}{}

\DeclareMathAlphabet{\mathcyr}{OT2}{cmr}{m}{l}
\DeclareMathAlphabet{\mathcyb}{OT2}{cmr}{b}{l}
\SetMathAlphabet{\mathcyr}{bold}{OT2}{cmr}{b}{l}

\usepackage[dvipdfmx]{hyperref}

\newtheorem{thm}{Theorem}[section]
\newtheorem{lem}[thm]{Lemma}
\newtheorem{prop}[thm]{Proposition}
\newtheorem{cor}[thm]{Corollary}
\newtheorem{conj}[thm]{Conjecture}
\theoremstyle{definition}

\theoremstyle{remark}

\newcommand{\sha}{\mathbin{\widetilde{\mathcyr{sh}}}}
\newcommand{\sh}{\mathbin{\mathcyr{sh}}}

\DeclareMathOperator{\reg}{reg}
\DeclareMathOperator{\im}{Im}

\begin{document}

\title[$t$-adic SMZVs for indices in which $1$ and $3$ appear alternately]{$t$-adic symmetric multiple zeta values for indices in which $1$ and $3$ appear alternately}

\author{Minoru Hirose}
\address[Minoru Hirose]{Institute for Advanced Research, Nagoya University, Furo-cho, Chikusa-ku, Nagoya, 464-8602, Japan}
\email{minoru.hirose@math.nagoya-u.ac.jp}

\author{Hideki Murahara}
\address[Hideki Murahara]{The University of Kitakyushu, 4-2-1 Kitagata, Kokuraminami-ku, Kitakyushu, Fukuoka, 802-8577, Japan}
\email{hmurahara@mathformula.page}

\author{Shingo Saito}
\address[Shingo Saito]{Faculty of Arts and Science, Kyushu University, 744, Motooka, Nishi-ku, Fukuoka, 819-0395, Japan}
\email{ssaito@artsci.kyushu-u.ac.jp}


\begin{abstract}
 We consider the symmetric multiple zeta values in $\mathcal{S}_m$ without modulo $\pi^2$ reduction for indices in which $1$ and $3$ appear alternately.
 We investigate those values that can be expressed as a polynomial of the Riemann zeta values,
 and give a conjecturally complete list of explicit formulas for such values.
\end{abstract}

\maketitle

\tableofcontents

\section{Introduction}
\subsection{Multiple zeta values and $t$-adic multiple zeta values}
An index is a finite (possibly empty) sequence of positive intgers.
We say that an index is admissible if either it is empty or its last component is greater than $1$.

For each admissible index $\boldsymbol{k}=(k_1,\dots,k_r)$, the multiple zeta value (MZV) is defined by
\[
 \zeta(\boldsymbol{k})=\zeta(k_1,\dots,k_r)=\sum_{1\le n_1<\dots<n_r}\frac{1}{n_1^{k_1}\dotsm n_r^{k_r}}\in\mathbb{R},
\]
where we understand that $\zeta(\emptyset)=1$.
Let $\mathcal{Z}$ denote the $\mathbb{Q}$-linear subspace of $\mathbb{R}$ spanned by all MZVs (including $1$).

For each (not necessarily admissible) index $\boldsymbol{k}$,
we write $\zeta^*(\boldsymbol{k})$ and $\zeta^{\sh}(\boldsymbol{k})$ for the real numbers in $\mathcal{Z}$
obtained by taking the constant terms of the harmonic and shuffle regularizations, respectively (see \cite{IKZ06} for details).
Note that $\zeta^*(1)=\zeta^{\sh}(1)=0$ and that $\zeta^*(\boldsymbol{k})=\zeta^{\sh}(\boldsymbol{k})=\zeta(\boldsymbol{k})$ if $\boldsymbol{k}$ is admissible.

For each index $\boldsymbol{k}=(k_1,\dots,k_r)$ and each symbol $\bullet\in\{*,\sh\}$, we define
\[
 \zeta_{m}^{\bullet}(\boldsymbol{k})
 =\sum_{\substack{l_1,\dots,l_r\ge0\\l_{1}+\dots+l_r=m}}\zeta^{\bullet}(k_1+l_1,\dots,k_r+l_r)\prod_{i=1}^{r}\binom{k_i+l_i-1}{l_i}
\]
for each nonnegative integer $m$, define the $t$-adic symmetric multiple zeta value ($t$-adic SMZV) by
\[
 \zeta_{\widehat{\mathcal{S}}}^{\bullet}(\boldsymbol{k})
 =\sum_{i=0}^{r}(-1)^{k_{i+1}+\dots+k_r}\zeta^{\bullet}(k_1,\dots,k_i)\sum_{m\ge0}\zeta_{m}^{\bullet}(k_r,\dots,k_{i+1})t^m\in\mathcal{Z}[[t]],
\]
and define the SMZV in $\mathcal{S}_m$ by
\[
 \zeta_{\mathcal{S}_m}^{\bullet}(\boldsymbol{k})=\pi_m(\zeta_{\widehat{\mathcal{S}}}^{\bullet}(\boldsymbol{k}))\in\mathcal{Z}[[t]]/t^m
\]
for each positive integer $m$, where $\pi_m\colon \mathcal{Z}[[t]]\to\mathcal{Z}[[t]]/t^m$ is the natural projection.
Here we understand that $\zeta_{\widehat{\mathcal{S}}}^{\bullet}(\emptyset)=1$.
The $t$-adic SMZVs have been studied in \cite{Jar19}, \cite{Ros19}, and \cite{OSY20} as a generalization of ordinary SMZVs (the $t$-adic SMZVs for $t=0$; see \cite{Kan19} and \cite{KZ19}).

In this paper, we concentrate on indices $\boldsymbol{k}$ in which $1$ and $3$ appear alternately, such as $(3)$, $(1,3,1)$, and $(3,1,3,1)$.
As a consequence, for all indices that appear in this paper, the values of $\zeta^*$ and $\zeta^{\sh}$ coincide, and so do the values of $\zeta_{\widehat{\mathcal{S}}}^*$ and $\zeta_{\widehat{\mathcal{S}}}^{\sh}$ (see Proposition~\ref{astsh});
we therefore often omit $*$ and $\sh$ hereinafter.

Let $\boldsymbol{k}$ be an index in which $1$ and $3$ appear alternately.
It is known that $\zeta(\boldsymbol{k})$ can be written as polynomials of the Riemann zeta values (see \cite{BBBL98}, \cite{BB03}, \cite{BY18}, and \cite{BC20}).
It turns out, however, that this is not always the case for $\zeta_{\widehat{\mathcal{S}}}(\boldsymbol{k})$.
Our results, to be described in precise terms in the next subsection,
explicitly write the value $\zeta_{\mathcal{S}_m}(\boldsymbol{k})$ as a polynomial of the Riemann zeta values for as large $m$ as practically possible.

We remark that our theorems and conjectures given in the next subsection concerning the $t$-adic SMZVs modulo $\pi^2$ may also well be valid for the $\boldsymbol{p}$-adic finite MZVs since the $t$-adic SMZVs modulo $\pi^2$ and the $\boldsymbol{p}$-adic finite MZVs are conjecturally isomorphic (for details, see \cite{Kan19}, \cite{KZ19}, and \cite{OSY20}; see also \cite{SW16} and \cite{MOS20} for previous works on $\boldsymbol{p}$-adic finite MZVs).

\subsection{Statements of our main theorems}
We now give precise statements of our main theorems and conjectures.
The proofs of the theorems will be given in the next section.

Let $\boldsymbol{k}$ be an index in which $1$ and $3$ appear alternately.
According to its first and last components, we divide into the following four cases:
$\boldsymbol{k}=(\{3,1\}^n,3),(\{1,3\}^n),(\{3,1\}^n),(\{1,3\}^n,1)$,
where $n$ is a nonnegative integer and $\{a,b\}^n$ denotes the $n$ times repetition of $a,b$,
e.g., $\{a,b\}^2=(a,b,a,b)$.

\subsubsection{Case of $\boldsymbol{k}=(\{3,1\}^n,3)$}
Consider the case where $\boldsymbol{k}$ both starts and ends with $3$, i.e., $\boldsymbol{k}$ is of the form $(\{3,1\}^n,3)$.
Then we have $\zeta_{\mathcal{S}_1}(\boldsymbol{k})=0$ by definition.
Since the coefficient of $t$ in $\zeta_{\mathcal{S}_2}(3,1,3)$ is congruent to
\[
 -5\zeta(3)\zeta(5)-\zeta(3,5)
\]
modulo $\pi^2$ (here and throughout, we have used \cite{BBV10} in numerical computations),
it is reasonable to believe that $\zeta_{\mathcal{S}_2}(3,1,3)$ cannot be written as a polynomial of the Riemann zeta values even when reduced modulo $\pi^2$, and so we do not investigate $\zeta_{\mathcal{S}_m}(\{3,1\}^n,3)$ for $m\ge2$ in this paper.

\subsubsection{Case of $\boldsymbol{k}=(\{1,3\}^n)$}
Consider the case where $\boldsymbol{k}$ starts with $1$ and ends with $3$, i.e., $\boldsymbol{k}$ is of the form $(\{1,3\}^n)$.
Then we can compute $\zeta_{\mathcal{S}_2}(\boldsymbol{k})$ explicitly as follows (remember that $\zeta(1)=0$):
\begin{thm} \label{main0}
 We have
 \begin{align*}
  \zeta_{\mathcal{S}_2}(\{1,3\}^n)
  &=\frac{2(-4)^n}{(4n+2)!}\pi^{4n}+\biggl(\sum_{\substack{n_0,n_1\ge0\\n_0+n_1=n}}\frac{(-4)^{n_0+1}(2-(-4)^{-n_1})}{(4n_0+2)!}\pi^{4n_0}\zeta(4n_1+1)\\
  &\hphantom{=\frac{2(-4)^n}{(4n+2)!}\pi^{4n}+\biggl(}-(-1)^n\sum_{\substack{n_0,n_1\ge0\\n_0+n_1=2n\\n_0,n_1:\mathrm{odd}}}\frac{2^{n_0-n_1+2}}{(2n_0+2)!}\pi^{2n_0}\zeta(2n_1+1)\biggr)t
 \end{align*}
 for every nonnegative integer $n$.
\end{thm}

Recall that Ono, Sakurada, and Seki \cite[Theorem~4.1]{OSS21} computed the values of $\zeta_{\mathcal{S}_2}$ modulo $\pi^2$ for a wider class of indices.
What makes Theorem~\ref{main0} interesting is that it computes the values without modulo $\pi^2$ reduction.

Since the coefficient of $t^2$ in $\zeta_{\mathcal{S}_3}(1,3,1,3)$ is
\[
 \frac{1}{2}\zeta(2) \zeta(3) \zeta(5)+\zeta(2) \zeta(3,5)-\frac{1}{2}\zeta(3)^2 \zeta(4)
 -\frac{1}{4}\zeta(3) \zeta(7)+\frac{81}{8}\zeta(5)^2 -\frac{103 }{10}\zeta(10),
\]
it is reasonable to believe that $\zeta_{\mathcal{S}_3}(1,3,1,3)$ cannot be written as a polynomial of the Riemann zeta values.
Nevertheless, numerical experiments suggest that $\zeta_{\mathcal{S}_3}(\{1,3\}^n)$ modulo $\pi^2$ can always be written as a polynomial of the Riemann zeta values, and we make the following conjecture:
\begin{conj}\label{conj:zeta_S3(1,3,1,3)}
 We have
 \begin{align*}
  \zeta_{\mathcal{S}_3}(\{1,3\}^n)
  &\equiv\delta_{n,0}+(2(-4)^{-n}-4)\zeta(4n+1)t\\
  &\hphantom{{}\equiv{}}+\biggl(-2(-4)^{-n}\sum_{\substack{n_1,n_2\ge0\\n_1+n_2=n-1}}\zeta(4n_1+3)\zeta(4n_2+3)\\
  &\hphantom{{}\equiv{}+\biggl(}+2\sum_{\substack{n_1,n_2\ge0\\n_1+n_2=n}}((-4)^{-n_1}-2)((-4)^{-n_2}-2)\zeta(4n_1+1)\zeta(4n_2+1)\biggr)t^2\pmod{\pi^2}
 \end{align*}
 for every nonnegative integer $n$, where $\delta_{n,0}$ denotes the Kronecker delta.
\end{conj}

Since the coefficient of $t^3$ in $\zeta_{\mathcal{S}_4}(1,3,1,3)$ is congruent to
\[
 -\frac{845}{4}\zeta(11)-\frac{9}{4}\zeta(3)^2\zeta(5)-\zeta(3)\zeta(3,5)+2\zeta(3,3,5)
\]
modulo $\pi^2$,
it is reasonable to believe that $\zeta_{\mathcal{S}_4}(1,3,1,3)$ cannot be written as a polynomial of the Riemann zeta values even when reduced modulo $\pi^2$, and so we do not investigate $\zeta_{\mathcal{S}_m}(\{1,3\}^n)$ for $m\ge4$ in this paper.

\subsubsection{Case of $\boldsymbol{k}=(\{3,1\}^n)$}
Consider the case where $\boldsymbol{k}$ starts with $3$ and ends with $1$, i.e., $\boldsymbol{k}$ is of the form $(\{3,1\}^n)$.
Then we can compute $\zeta_{\mathcal{S}_3}(\boldsymbol{k})$ explicitly as follows:
\begin{thm} \label{main1}
 We have
 \begin{align*}
  \zeta_{\mathcal{S}_3}(\{3,1\}^n)
  &=\frac{2(-4)^n}{(4n+2)!}\pi^{4n}+(-1)^{n+1}\sum_{\substack{n_0,n_1\ge0\\n_0+n_1=2n}}\frac{(-1)^{n_0}2^{n_0-n_1+2}}{(2n_0+2)!}\pi^{2n_0}\zeta(2n_1+1)t\\
  &\hphantom{{}={}}+(-1)^n\sum_{\substack{n_0,n_1,n_2\ge0\\n_0+n_1+n_2=2n}}\frac{(-1)^{n_0}2^{n_0-n_1-n_2+2}}{(2n_0+2)!}\pi^{2n_0}\zeta(2n_1+1)\zeta(2n_2+1)t^2
 \end{align*}
 for every nonnegative integer $n$.
\end{thm}

\begin{cor} \label{cor1}
 We have
 \begin{align*}
  &\zeta_{\mathcal{S}_3}(\{3,1\}^n)\\
  &\equiv\delta_{n,0}-2(-4)^{-n}\zeta(4n+1)t
  +2(-4)^{-n}\sum_{\substack{n_1,n_2\ge0\\n_1+n_2=2n}}\zeta(2n_1+1)\zeta(2n_2+1)t^2\pmod{\pi^2}
 \end{align*}
 for every nonnegative integer $n$.
\end{cor}

Since the coefficient of $t^3$ in $\zeta_{\mathcal{S}_4}(3,1,3,1)$ is congruent to
\[
 \frac{605}{4}\zeta(11)+\frac{19}{4}\zeta(3)^2\zeta(5)+2\zeta(3)\zeta(3,5)-2\zeta(3,3,5)
\]
modulo $\pi^2$,
it is reasonable to believe that $\zeta_{\mathcal{S}_4}(3,1,3,1)$ cannot be written as a polynomial of the Riemann zeta values even when reduced modulo $\pi^2$, and so we do not investigate $\zeta_{\mathcal{S}_m}(\{3,1\}^n)$ for $m\ge4$ in this paper.

\subsubsection{Case of $\boldsymbol{k}=(\{1,3\}^n,1)$}
Consider the case where $\boldsymbol{k}$ both starts and ends with $1$, i.e., $\boldsymbol{k}$ is of the form $(\{1,3\}^n,1)$.
Then we can compute $\zeta_{\mathcal{S}_3}(\boldsymbol{k})$ explicitly as follows:
\begin{thm} \label{main2}
 We have
 \[
  \zeta_{\mathcal{S}_3}(\{1,3\}^n,1)
  =\frac{(-4)^{n+1}}{(4n+4)!}\pi^{4n+2}t
   +(-1)^n\sum_{\substack{n_0,n_1\ge0\\n_0+n_1=2n+1}}\frac{(-1)^{n_1}2^{n_0-n_1+2}}{(2n_0+2)!}\pi^{2n_0}\zeta(2n_1+1)t^2
 \]
 for every nonnegative integer $n$.
\end{thm}

\begin{cor} \label{cor2}
 We have
 \[
  \zeta_{\mathcal{S}_3}(\{1,3\}^n,1)\equiv-(-4)^{-n}\zeta(4n+3)t^2\pmod{\pi^2}
 \]
 for every nonnegative integer $n$.
\end{cor}

Since the coefficient of $t^3$ in $\zeta_{\mathcal{S}_4}(1,3,1)$ is congruent to
\[
 \frac{9}{2}\zeta(3)\zeta(5)+\zeta(3,5)
\]
modulo $\pi^2$,
it is reasonable to believe that $\zeta_{\mathcal{S}_4}(1,3,1)$ cannot be written as a polynomial of the Riemann zeta values even when reduced modulo $\pi^2$, and so we do not investigate $\zeta_{\mathcal{S}_m}(\{1,3\}^n,1)$ for $m\ge4$ in this paper.

\subsubsection{Summary}
\begin{conj}
 The pairs $(\boldsymbol{k},m)$ of an index $\boldsymbol{k}$ in which $1$ and $3$ appear alternately and a positive integer $m$
 such that $\zeta_{\mathcal{S}_m}(\boldsymbol{k})$ can be written as a polynomial of the Riemann zeta values
 are exhausted by those deduced from Theorems~\ref{main0}, \ref{main1}, and \ref{main2} and the following equations:
 \begin{gather*}
  \zeta_{\mathcal{S}_1}(\{3,1\}^n,3)=0\qquad(n\ge0),\\
  \zeta_{\widehat{\mathcal{S}}}(1)=-\sum_{m\ge1}\zeta(m+1)t^m,\qquad
  \zeta_{\widehat{\mathcal{S}}}(3)=-\frac{1}{2}\sum_{m\ge1}(m+1)(m+2)\zeta(m+3)t^m,\\
  \begin{aligned}
   \zeta_{\mathcal{S}_3}(1,3)&=-\frac{\pi^4}{90}+\biggl(\frac{\pi^2}{6}\zeta(3)-\frac{9}{2}\zeta(5)\biggr)t+\biggl(-\frac{19\pi^6}{3780}+\frac{1}{2}\zeta(3)^2\biggr)t^2+\biggl(\frac{\pi^4}{90}\zeta(3)+\pi^2\zeta(5)-17\zeta(7)\biggr)t^3,\\
   \zeta_{\mathcal{S}_3}(3,1)&=-\frac{\pi^4}{90}+\biggl(-\frac{\pi^2}{6}\zeta(3)+\frac{1}{2}\zeta(5)\biggr)t-\frac{1}{2}\zeta(3)^2t^2+\biggl(\frac{\pi^4}{45}\zeta(3)-3\zeta(7)\biggr)t^3.
  \end{aligned}
 \end{gather*}
\end{conj}

\section{Proofs of our main theorems} \label{sec2}
\subsection{Algebraic setup}
We use Hoffman's algebraic setup with a slightly different convention (see \cite{Hof97}).
Set $\mathfrak{H}=\mathbb{Q}\langle x,y\rangle$, $\mathfrak{H}^1=\mathbb{Q}+y\mathfrak{H}$, and $\mathfrak{H}^0=\mathbb{Q}+y\mathfrak{H}x$.
We define the shuffle product as the $\mathbb{Q}$-bilinear product $\sh\colon\mathfrak{H}\times\mathfrak{H}\to\mathfrak{H}$ given by
\begin{align*}
 1\sh w&=w\sh 1=w, \\
 uw \sh u'w' &=u(w\sh u'w') +u'(uw\sh w'),
\end{align*}
where $w,w'\in\mathfrak{H}$ and $u,u'\in\{x,y\}$.
This product makes $\mathfrak{H}$ a commutative $\mathbb{Q}$-algebra, which we denote by $\mathfrak{H}_{\sh}$ (see \cite{Reu93}).
The subspaces $\mathfrak{H}^1$ and $\mathfrak{H}^0$ become subalgebras of $\mathfrak{H}_{\sh}$, which we denote by $\mathfrak{H}^1_{\sh}$ and $\mathfrak{H}^0_{\sh}$, respectively.

For a positive integer $k$, put $z_k=yx^{k-1}$.
We define the $\mathbb{Q}$-linear map $Z\colon \mathfrak{H}^0 \to \mathbb{R}$ by
\[
 Z (z_{k_1} \cdots z_{k_r})
 =\zeta (k_1,\dots,k_r), 
\]
where $(k_1,\dots,k_r)$ is an admissible index.
Note that $Z(w_1\sh w_2)=Z(w_1) Z(w_2)$ holds for $w_1,w_2\in \mathfrak{H}^0$.

We define the algebra homomorphism $\reg_{\sh} \colon \mathfrak{H}_{\sh}\to \mathfrak{H}^0_{\sh}$ by the properties that it is the identity on $\mathfrak{H}^0$, maps $x$ to $0$, and maps $y$ to $0$.
We also define $Z^{\sh}\colon \mathfrak{H}_{\sh}\to \mathbb{R}$ by
 \[
  Z^{\sh}= Z \circ \reg_{\sh}.
 \]

\begin{prop} \label{regshwd}
 For a nonnegative integer $m$ and positive integers $k_1,\dots,k_r$, we have
 \begin{align*}
  x^m yx^{k_1-1} \cdots yx^{k_r-1}
  \equiv (-1)^m \sum_{\substack{ l_1,\dots,l_{r}\ge0 \\ l_1+\cdots+l_{r}=m }}
   yx^{k_1+l_1-1} \cdots yx^{k_r+l_r-1}
   \prod_{j=1}^{r} \binom{k_j+l_j-1}{l_j}
 \end{align*}
 modulo $x \sh \mathfrak{H}$.
\end{prop}

\begin{proof}
 Note that induction shows
 \[
  x^{m}yw = \sum_{i=0}^m (-1)^{m-i} x^i \sh y(x^{m-i}\sh w)
 \]
 for all $w\in \mathfrak{H}$.
 Setting $w=x^{k_1-1} y x^{k_2-1}  \cdots yx^{k_r-1}$, we have
 \begin{align*}
  x^m yx^{k_1-1} \cdots yx^{k_r-1}
  &=x^m yw
  \equiv (-1)^m y(x^m \sh w) \\
  &=(-1)^m \sum_{\substack{ l_1,\dots,l_{r}\ge0 \\ l_1+\cdots+l_{r}=m }}
   yx^{k_1+l_1-1} \cdots yx^{k_r+l_r-1}
   \prod_{j=1}^{r} \binom{k_j+l_j-1}{l_j}
 \end{align*}
 modulo $x \sh \mathfrak{H}$.
\end{proof}

From the previous proposition, we immediately have the following.
\begin{prop} \label{zmZ}
 For a nonnegative integer $m$ and positive integers $k_{1},\dots,k_{r}$, we have
 \begin{align*}
  \zeta_m^{\sh} (k_{1},\dots,k_{r})
  =(-1)^m Z^{\sh} (x^myx^{k_1-1}\cdots yx^{k_r-1}).
 \end{align*}
\end{prop}

\begin{prop} \label{astsh}
 Let $\boldsymbol{k}$ be an index with no adjacent ones.
 Then we have
 \begin{align*}
  \zeta^{\ast} (\boldsymbol{k})
  =\zeta^{\sh} (\boldsymbol{k}),
  \quad
  \zeta_{\widehat{\mathcal{S}}}^{\ast} (\boldsymbol{k})
  =\zeta_{\widehat{\mathcal{S}}}^{\sh} (\boldsymbol{k}).
 \end{align*}
\end{prop}

\begin{proof}
 By \cite[Theorem 1]{IKZ06}, we have
 \[
  \zeta^{\ast} (\boldsymbol{l},1)
  =\zeta^{\sh} (\boldsymbol{l},1)
 \]
 for an index $\boldsymbol{l}$ whose last component is grater than $1$.
 Then, by definitions, we easily see the result.
\end{proof}

Let $\tau$ be the anti-automorphism on $\mathfrak{H}$ with $\tau(x)=y$ and $\tau(y)=x$.
Then the duality formula of MZVs says $Z(w)=Z(\tau(w))$ for $w\in\mathfrak{H}^0$, which is generalized as follows.
\begin{prop} \label{duality}
 For $w\in\mathfrak{H}$, we have
 \[
  Z^{\sh}(w)=Z^{\sh}(\tau(w)).
 \]
\end{prop}

\begin{proof}
 Since $\reg_{\sh}$ and $\tau$ commute, we have
 \begin{align*}
  Z^{\sh}(w)=Z(\reg_{\sh}(w))=Z(\tau(\reg_{\sh}(w)))=Z(\reg_{\sh}(\tau(w)))=Z^{\sh}(\tau(w)),
 \end{align*}
 as required.
\end{proof}

\subsection{Alternating sums of shuffle products}
In this subsection, we compute several alternating sums of shuffle products to be used later in the proof of the main theorems.\begin{lem} \label{wordA}
 For a nonnegative integer $n$, we have
 \[
  \sum_{i=0}^{n}(-1)^ix(yx)^i\sh(yx)^{n-i}=
  \begin{cases}
   (-1)^{n/2}2^nx(y^2x^2)^{n/2}&\textrm{if $n$ is even},\\
   (-1)^{(n-1)/2}2^ny(x^2y^2)^{(n-1)/2}x^2&\textrm{if $n$ is odd}.
  \end{cases}
 \]
\end{lem}

\begin{proof}
 Put $a_{n}=\sum_{i=0}^{n}(-1)^{i}x(yx)^{i}\sh(yx)^{n-i}$.
 If $n\ge1$,  we have
 \begin{align*}
  a_{n}
  &=xy\sum_{i=1}^{n}(-1)^{i}x(yx)^{i-1}\sh(yx)^{n-i}
   +xy\sum_{i=0}^{n-1}(-1)^{i}(yx)^{i}\sh x(yx)^{n-i-1}\\
  &\quad +yx\sum_{i=0}^{n-1}(-1)^{i}(yx)^{i}\sh x(yx)^{n-i-1}
   +yx\sum_{i=0}^{n-1}(-1)^{i}x(yx)^{i}\sh(yx)^{n-i-1} \\
  &=( -(1+(-1)^n)xy +(1-(-1)^{n})yx ) a_{n-1}.
 \end{align*}
 Since $a_0=x$, we obtain the result by induction on $n$.
\end{proof}

\begin{lem} \label{word1}
 For a nonnegative integer $n$, we have
 \[
  \sum_{i=0}^{n}(-1)^ix(yx)^i\sh(xy)^{n-i}=
  \begin{cases}
   (-1)^{n/2}2^n(x^2y^2)^{n/2} x
    &\textrm{if $n$ is even}, \\
   (-1)^{(n-1)/2} 2^n x^2(y^2x^2)^{(n-1)/2}y
    &\textrm{if $n$ is odd}.
  \end{cases}
 \]
\end{lem}

\begin{proof}
 By looking at the reversal of Lemma \ref{wordA}, we find the result.
\end{proof}

\begin{lem} \label{wordB}
 For a nonnegative integer $n$, we have
 \begin{align*}
  &\sum_{i=0}^{n} (-1)^i x(yx)^i \sh (yx)^{n-i}y \\
  &=
  \begin{cases}
   (-1)^{n/2} 2^n
   ( x(y^2x^2)^{n/2}y +(yx^2y)^{n/2}yx )
   &\textrm{if $n$ is even}, \\
   (-1)^{(n-1)/2} 2^n
   ( -xy^2(x^2y^2)^{(n-1)/2}x +y(x^2y^2)^{(n-1)/2}x^2y )
   &\textrm{if $n$ is odd}.
  \end{cases}
 \end{align*}
\end{lem}

\begin{proof}
Let $a_{n}$ and $a_{n}'$ be the left-hand sides of Lemmas \ref{wordA} and \ref{word1}, respectively.
Then
\[
\sum_{i=0}^{n}(-1)^{i}x(yx)^{i}\sh(yx)^{n-i}y=(-1)^{n}x\tau(a_{n})+ya_{n}'.
\]
Thus the claim follows from Lemmas \ref{wordA} and \ref{word1}.
\end{proof}

\begin{lem} \label{word2}
 For a nonnegative integer $n$, we have
 \begin{align*}
  \sum_{i=0}^{n} (-1)^i x(yx)^i \sh x(yx)^{n-i}
  &=
  \begin{cases}
   (-1)^{n/2} 2^{n+1} x^2(y^2x^2)^{n/2}
    &\textrm{if $n$ is even}, \\
   0
    &\textrm{if $n$ is odd}.
  \end{cases}
 \end{align*}
\end{lem}

\begin{proof}
Let $a_{n}$ be the left-hand side of Lemma \ref{wordA}. Then
\[
\sum_{i=0}^{n}(-1)^{i}x(yx)^{i}\sh x(yx)^{n-i}=(1+(-1)^{n})xa_{n}.
\]
Thus the claim follows from Lemma \ref{wordA}.
\end{proof}

\subsection{Harmonic product and multiple zeta-star values}
Let $\mathcal{I}$ be the $\mathbb{Q}$-vector space freely generated by all indices.
We define the $\mathbb{Q}$-bilinear product $\ast$ on $\mathcal{I}$ inductively by setting
\begin{align*}
 \boldsymbol{k} \ast \emptyset
 &=\emptyset \ast \boldsymbol{k}=\boldsymbol{k}, \\
  (\boldsymbol{k},k) \ast (\boldsymbol{l},l)
 &=(\boldsymbol{k} \ast (\boldsymbol{l},l),k)
  +(\boldsymbol{k} \ast \boldsymbol{l},k+l)
  +((\boldsymbol{k},k) \ast \boldsymbol{l},l)
\end{align*}
for all indices $\boldsymbol{k}$, $\boldsymbol{l}$ and all positive integers $k$, $l$.
Note that $\zeta(\boldsymbol{k} \ast \boldsymbol{l})=\zeta(\boldsymbol{k})\zeta(\boldsymbol{l})$ holds for any indices $\boldsymbol{k}$ and $\boldsymbol{l}$ whose last component is grater than $1$.

For each admissible index $\boldsymbol{k}=(k_1,\dots,k_r)$, the multiple zeta-star value (MZSV) is defined by
\[
 \zeta^{\star}(\boldsymbol{k})=\zeta^{\star}(k_1,\dots,k_r)=\sum_{1\le n_1\le\dots\le n_r}\frac{1}{n_1^{k_1}\dotsm n_r^{k_r}}\in\mathbb{R},
\]
where we understand that $\zeta^{\star}(\emptyset)=1$.
It is well known that
\[
 \sum_{i=0}^{r}(-1)^i\zeta(k_1,\dots,k_i)\zeta^{\star}(k_r,\dots,k_{i+1})=\delta_{r,0}
\]
for every index $\boldsymbol{k}=(k_1,\dots,k_r)$ satisfying $k_1,\dots,k_r\ge2$, which is sometimes referred to as the antipode formula.

\subsection{Computation of generating series}
In this subsection, we compute several generating series to be used later in the proof of our main theorems.
It turns out all generating series that we will need can be expressed in terms of the four generating series $F_+,F_-,G_+,G_-\in\mathbb{R}[[u]]$ defined by
\[
 F_{\pm}=\sum_{n\ge0}(\pm2)^{-n}\zeta(2n+1)u^{2n+1},\qquad
 G_{\pm}=\sum_{n\ge0}(\pm2)^{-n}\zeta(\{2\}^n)u^{2n}.
\]

Observe that
\begin{align*}
 F_++F_-&=2\sum_{\substack{n\ge0\\n:\mathrm{even}}}2^{-n}\zeta(2n+1)u^{2n+1}=2\sum_{n\ge0}4^{-n}\zeta(4n+1)u^{4n+1},\\
 F_+-F_-&=2\sum_{\substack{n\ge0\\n:\mathrm{odd}}}2^{-n}\zeta(2n+1)u^{2n+1}=\sum_{n\ge0}4^{-n}\zeta(4n+3)u^{4n+3},\\
 G_++G_-&=2\sum_{\substack{n\ge0\\n:\mathrm{even}}}2^{-n}\zeta(\{2\}^n)u^{2n}=2\sum_{n\ge0}4^{-n}\zeta(\{2\}^{2n})u^{4n},\\
 G_+-G_-&=2\sum_{\substack{n\ge0\\n:\mathrm{odd}}}2^{-n}\zeta(\{2\}^n)u^{2n}=\sum_{n\ge0}4^{-n}\zeta(\{2\}^{2n+1})u^{4n+2}
\end{align*}
and that
\[
 \sum_{n\ge0}(\pm2)^{-n}\zeta^{\star}(\{2\}^n)u^{2n}=G_{\mp}^{-1}
\]
by the antipode formula.

\begin{lem}\label{lem:zeta(4^n)_gen}
 We have
 \[
  G_+G_-
  =\sum_{n\ge0}(-1)^n\zeta(\{1,3\}^n)u^{4n}
  =\sum_{n=0}^{\infty}(-4)^{-n}\zeta(\{4\}^n)u^{4n}
  =\biggl(\sum_{n=0}^{\infty}4^{-n}\zeta^{\star}(\{4\}^n)u^{4n}\biggr)^{-1}
 \]
 and
 \[
  G_{\pm}^2=\sum_{n\ge0}\frac{(\pm1)^n2^{n+1}\pi^{2n}}{(2n+2)!}u^{2n},\qquad
  \frac{G_+^2+G_-^2}{2}=\sum_{n\ge0}\zeta(\{4\}^n)u^{4n}.
 \]
\end{lem}

\begin{proof}
 Recall that
 \[
  \zeta(\{2\}^n)=\frac{\pi^{2n}}{(2n+1)!},\qquad
  \zeta(\{1,3\}^n)=\frac{2\pi^{4n}}{(4n+2)!},\qquad
  \zeta(\{4\}^n)=\frac{2^{2n+1}\pi^{4n}}{(4n+2)!}
 \]
 for every nonnegative integer $n$ (see Borwein, Bradley, Broadhurst, and Lison\v{e}k \cite[Example 2.2]{BBBL98} for the last two identities).
 It follows that
 \begin{align*}
  G_+G_-
  &=\biggl(\sum_{n_1\ge0}\frac{1}{(2n_1+1)!}\biggl(\frac{\pi^2u^2}{2}\biggr)^{n_1}\biggl)
    \biggl(\sum_{n_2\ge0}\frac{(-1)^{n_2}}{(2n_2+1)!}\biggl(\frac{\pi^2u^2}{2}\biggr)^{n_2}\biggr)\\
  &=\sum_{n\ge0}\biggl(\sum_{\substack{n_1+n_2=n\\n_1,n_2\ge0}}\frac{(-1)^{n_2}}{(2n_1+1)!(2n_2+1)!}\biggr)\biggl(\frac{\pi^2u^2}{2}\biggr)^n\\
  &=\sum_{n\ge0}\frac{\im((1+i)^{2n+2})}{(2n+2)!}\biggl(\frac{\pi^2u^2}{2}\biggr)^n\qquad(\text{binomial theorem})\\
  &=\sum_{\substack{n\ge0\\n:\mathrm{even}}}\frac{(-1)^{n/2}2^{n+1}}{(2n+2)!}\biggl(\frac{\pi^2u^2}{2}\biggr)^n\\
  &=\sum_{n\ge0}\frac{2(-1)^{n}\pi^{4n}}{(4n+2)!}u^{4n},
 \end{align*}
 which together with the antipode formula implies that
 \[
  G_+G_-
  =\sum_{n\ge0}(-1)^n\zeta(\{1,3\}^n)u^{4n}
  =\sum_{n=0}^{\infty}(-4)^{-n}\zeta(\{4\}^n)u^{4n}
  =\biggl(\sum_{n=0}^{\infty}4^{-n}\zeta^{\star}(\{4\}^n)u^{4n}\biggr)^{-1}.
 \]
 We also have
 \begin{align*}
  G_{\pm}^2
  &=\biggl(\sum_{n_1\ge0}\frac{1}{(2n_1+1)!}\biggl(\pm\frac{\pi^2u^2}{2}\biggr)^{n_1}\biggl)
    \biggl(\sum_{n_2\ge0}\frac{1}{(2n_2+1)!}\biggl(\pm\frac{\pi^2u^2}{2}\biggr)^{n_2}\biggr)\\
  &=\sum_{n\ge0}\biggl(\sum_{\substack{n_1+n_2=n\\n_1,n_2\ge0}}\frac{1}{(2n_1+1)!(2n_2+1)!}\biggr)\biggl(\pm\frac{\pi^2u^2}{2}\biggr)^n\\
  &=\sum_{n\ge0}\frac{2^{2n+1}}{(2n+2)!}\biggl(\pm\frac{\pi^2u^2}{2}\biggr)^n\qquad(\text{binomial theorem})\\
  &=\sum_{n\ge0}\frac{(\pm1)^n2^{n+1}\pi^{2n}}{(2n+2)!}u^{2n},
 \end{align*}
 which implies that
 \[
  \frac{G_+^2+G_-^2}{2}
  =\sum_{\substack{n\ge0\\n:\mathrm{even}}}\frac{2^{n+1}\pi^{2n}}{(2n+2)!}u^{2n}
  =\sum_{\substack{n\ge0}}\frac{2^{2n+1}\pi^{4n}}{(4n+2)!}u^{4n}
  =\sum_{\substack{n\ge0}}\zeta(\{4\}^n)u^{4n}.\qedhere
 \]
\end{proof}

\begin{lem}\label{lem:zeta(3,1,3)_gen}
 We have
 \[
  \sum_{n\ge0}(-1)^n\zeta(\{3,1\}^n,3)u^{4n+3}=(F_+-F_-)G_+G_-.
 \]
\end{lem}

\begin{proof}
 Recall that Bowman and Bradley \cite[Theorem~1]{BB03} showed that
 \[
  \zeta(\{3,1\}^n,3)=4^{-n}\sum_{i=0}^{n}(-1)^i\zeta(4i+3)\zeta(\{4\}^{n-i})
 \]
 for every nonnegative integer $n$.
 It follows that
 \begin{align*}
  \sum_{n\ge0}(-1)^n\zeta(\{3,1\}^n,3)u^{4n+3}
  &=\sum_{n\ge0}(-4)^{-n}\sum_{i=0}^{n}(-1)^i\zeta(4i+3)\zeta(\{4\}^{n-i})u^{4n+3}\\
  &=\biggl(\sum_{n_1\ge0}4^{-n_1}\zeta(4n_1+3)u^{4n_1+3}\biggr)\biggl(\sum_{n_2\ge0}(-4)^{-n_2}\zeta(\{4\}^{n_2})u^{4n_2}\biggr)\\
  &=(F_+-F_-)G_+G_-
 \end{align*}
 by Lemma~\ref{lem:zeta(4^n)_gen}.
\end{proof}

\begin{lem}\label{lem:zeta(1,3,1)_gen}
 We have
 \[
  \sum_{n\ge0}(-1)^n\zeta(\{1,3\}^n,1)u^{4n+1}=(F_++F_-)G_+G_-.
 \]
\end{lem}

\begin{proof}
 Recall that Bachmann and Charlton \cite[Proposition 4.2]{BC20} showed that
 \[
  \zeta(\{1,3\}^n,1)=2^{-2n+1}\sum_{i=0}^{n}(-1)^i\zeta(4i+1)\zeta(\{4\}^{n-i})
 \]
 for every nonnegative integer $n$ (note that the summand is equal to $0$ if $i=0$).
 It follows that
 \begin{align*}
  \sum_{n\ge0}(-1)^n\zeta(\{1,3\}^n,1)u^{4n+1}
  &=2\sum_{n\ge0}(-4)^{-n}\sum_{i=0}^{n}(-1)^i\zeta(4i+1)\zeta(\{4\}^{n-i})u^{4n+1}\\
  &=\biggl(2\sum_{n_1\ge0}4^{-n_1}\zeta(4n_1+1)u^{4n_1+1}\biggr)\biggl(\sum_{n_2\ge0}(-4)^{-n_2}\zeta(\{4\}^{n_2})u^{4n_2}\biggr)\\
  &=(F_++F_-)G_+G_-
 \end{align*}
 by Lemma~\ref{lem:zeta(4^n)_gen}.
\end{proof}

\begin{lem}\label{lem:zeta(3,1,3,1)_gen}
 We have
 \[
  \sum_{n\ge0}(-1)^n\zeta(\{3,1\}^n)u^{4n}=\frac{G_+^2+G_-^2}{2G_+G_-}-(F_+^2-F_-^2)G_+G_-.
 \]
\end{lem}

\begin{proof}
 Recall that Bachmann and Charlton \cite[Proposition 4.2]{BC20} showed that
 \begin{align*}
  \zeta(\{3,1\}^n)
  &=(-1)^n\sum_{i=0}^{n}4^{-i}\zeta^\star(\{4\}^i)\zeta(\{4\}^{n-i})\\
  &\hphantom{{}={}}+2^{-2n+3}\sum_{\substack{1\le i\le n-1\\0\le j\le n-i-1}}(-1)^{i+j}\zeta(4i+1)\zeta(4j+3)\zeta(\{4\}^{n-i-j-1})
 \end{align*}
 for every nonnegative integer $n$.
 It follows that
 \begin{align*}
  &\sum_{n\ge0}(-1)^n\zeta(\{3,1\}^n)u^{4n}\\
  &=\sum_{n\ge0}\sum_{i=0}^{n}4^{-i}\zeta^\star(\{4\}^i)\zeta(\{4\}^{n-i})u^{4n}\\
  &\hphantom{{}={}}-2\sum_{n\ge0}(-4)^{-(n-1)}\sum_{\substack{1\le i\le n-1\\0\le j\le n-i-1}}(-1)^{i+j}\zeta(4i+1)\zeta(4j+3)\zeta(\{4\}^{n-i-j-1})\\
  &=\biggl(\sum_{n_1\ge0}4^{-n_1}\zeta^{\star}(\{4\}^{n_1})u^{4n_1}\biggr)\biggl(\sum_{n_2\ge0}\zeta(\{4\}^{n_2})u^{4n_2}\biggr)\\
  &\hphantom{{}={}}-\biggl(2\sum_{n_1\ge0}4^{-n_1}\zeta(4n_1+1)u^{4n_1+1}\biggr)\biggl(\sum_{n_2\ge0}4^{-n_2}\zeta(4n_2+3)u^{4n_2+3}\biggr)\biggl(\sum_{n_3\ge0}(-4)^{-n_3}\zeta(\{4\}^{n_3})u^{4n_3}\biggr)\\
  &=G_+^{-1}G_-^{-1}\frac{G_+^2+G_-^2}{2}-(F_++F_-)(F_+-F_-)G_+G_-\qquad(\text{Lemma~\ref{lem:zeta(4^n)_gen}})\\
  &=\frac{G_+^2+G_-^2}{2G_+G_-}-(F_+^2-F_-^2)G_+G_-.\qedhere
 \end{align*}
\end{proof}

\begin{lem}\label{prop:2star_2star}
 We have
 \[
 \sum_{n\ge0}4^{-n}(4n+1)\zeta(4n+2)u^{4n+2}=\frac{G_-^{-2}-G_+^{-2}}{2}.
 \]
\end{lem}

\begin{proof}
 Since
 \begin{align*}
  \bigg(\sum_{n\ge0}\zeta^{\star}(\{2\}^n)v^{2n}\bigg)^2
  &=\bigg(\frac{\pi v}{\sin(\pi v)}\bigg)^{2}
   =-\pi v^{2}\frac{d}{dv}\cot(\pi v)
   =2v^{2}\frac{d}{dv}\sum_{n\ge0} \zeta(2n)v^{2n-1}\\
  &=2\sum_{n\ge0} (2n-1)\zeta(2n)v^{2n}
 \end{align*}
 (here we understand $\zeta(0)=-1/2$),
 setting $v^2=\mp u^2/2$ we have
 \[
  G_{\pm}^{-2}=2\sum_{n\ge0}(\mp2)^{-n}(2n-1)\zeta(2n)u^{2n}.
 \]
 Then we get
 \[
  G_-^{-2}-G_+^{-2}=4\sum_{\substack{n\ge0\\n:\mathrm{odd}}}2^{-n}(2n-1)\zeta(2n)u^{2n}=2\sum_{n\ge0}4^{-n}(4n+1)\zeta(4n+2)u^{4n+2},
 \]
 as required.  
\end{proof}

\begin{lem}\label{lem:zeta(3,1,2)_gen}
 We have
 \[
  \sum_{n\ge0}(-1)^n\zeta(\{3,1\}^n,2)u^{4n+2}=\frac{G_+^2-G_-^2}{2G_+G_-}-(F_+-F_-)^2G_+G_-.
 \]
\end{lem}

\begin{proof}
 Recall that Bowman and Bradley \cite[Theorem~2]{BB03} showed that
 \[
  \zeta(\{3,1\}^n,2)=4^{-n}\sum_{i=0}^{n}(-1)^i\zeta(\{4\}^{n-i})\biggl((4i+1)\zeta(4i+2)-4\sum_{j=1}^i\zeta(4j-1)\zeta(4i-4j+3)\biggr)
 \]
 for every nonnegative integer $n$.
 It follows that
 \begin{align*}
  &\sum_{n\ge0}(-1)^n\zeta(\{3,1\}^n,2)u^{4n+2}\\
  &=\sum_{n\ge0}(-4)^{-n}\sum_{i=0}^{n}(-1)^i\zeta(\{4\}^{n-i})(4i+1)\zeta(4i+2)u^{4n+2}\\
  &\phantom{{}={}}-4\sum_{n\ge0}(-4)^{-n}\sum_{i=0}^{n}(-1)^i\zeta(\{4\}^{n-i})\sum_{j=1}^i\zeta(4j-1)\zeta(4i-4j+3)u^{4n+2}\\
  &=\biggl(\sum_{n_1\ge0}(-4)^{-n_1}\zeta(\{4\}^{n_1})u^{4n_1}\biggr)\biggl(\sum_{n_2\ge0}4^{-n_2}(4n_2+1)\zeta(4n_2+2)u^{4n_2+2}\biggr)\\
  &\phantom{{}={}}-\biggl(\sum_{n_1\ge0}(-4)^{-n_1}\zeta(\{4\}^{n_1})u^{4n_1}\biggr)\biggl(\sum_{n_2\ge0}4^{-n_2}\zeta(4n_2+3)u^{4n_2+3}\biggr)\biggl(\sum_{n_3\ge0}4^{-n_3}\zeta(4n_3+3)u^{4n_3+3}\biggr)\\
  &=G_+G_-\cdot\frac{G_-^{-2}-G_+^{-2}}{2}-G_+G_-(F_+-F_-)^2\qquad(\text{Lemmas \ref{lem:zeta(4^n)_gen} and \ref{prop:2star_2star}})\\
  &=\frac{G_+^2-G_-^2}{2G_+G_-}-(F_+-F_-)^2G_+G_-.\qedhere
 \end{align*}
\end{proof}

\begin{lem}\label{lem:zeta_1(1,3,1,3)_gen}
 We have
 \[
  \sum_{n\ge0}(-1)^n\zeta_1(\{1,3\}^n)u^{4n+1}=-(F_++F_-)G_+G_-.
 \]
\end{lem}

\begin{proof}
 Since Propositions \ref{zmZ} and \ref{duality} show that
 \[
  \zeta_1(\{1,3\}^n)=-Z^{\sh}(x(y^2x^2)^n)=-Z^{\sh}((y^2x^2)^ny)=-\zeta(\{1,3\}^n,1)
 \]
 for every nonnegative integer $n$, the lemma follows from Lemma~\ref{lem:zeta(1,3,1)_gen}.
\end{proof}

Let $\sha$ denote the shuffle of indices, e.g., $(a,b)\sha (c)=(a,b,c)+(a,c,b)+(c,a,b)$.

\begin{lem}\label{lem:sha_alternating_sum}
 If $a$ and $b$ are positive integers, then we have
 \[
  \sum_{i=0}^{n}(-1)^i(ai+b)*(\{a\}^{n-i})=(b)\sha(\{a\}^n)
 \]
 for every nonnegative integer $n$.
\end{lem}

\begin{proof}
 We have
 \begin{align*}
  &\sum_{i=0}^{n}(-1)^i(ai+b)*(\{a\}^{n-i})\\
  &=\sum_{i=0}^{n-1}(-1)^i((ai+b)\sha(\{a\}^{n-i})+(a(i+1)+b)\sha(\{a\}^{n-i-1}))+(-1)^n(an+b)\\
  &=(b)\sha(\{a\}^n),
 \end{align*}
 as required.
\end{proof}

\begin{lem}\label{lem:zeta_1(2^n)}
 We have
 \[
  \zeta_1(\{2\}^n)=-2 \sum_{i=0}^{n} (-1)^{i} \zeta(2i+1) \zeta(\{2\}^{n-i})
 \]
 for every nonnegative integer $n$.
\end{lem}

\begin{proof}
 We have
 \begin{align*}
  \zeta_1(\{2\}^n)
  &=\binom{2+1-1}{1}\sum_{i=1}^{n}\zeta(\{2\}^{i-1},3,\{2\}^{n-i})\\
  &=2\zeta((1)*(\{2\}^n)-(1)\sha(\{2\}^n))\\
  &=-2\zeta((1)\sha(\{2\}^n))\\
  &=-2\sum_{i=0}^{n}(-1)^i\zeta(2i+1)\zeta(\{2\}^{n-i})
 \end{align*}
 by Lemma~\ref{lem:sha_alternating_sum} with $a=2$ and $b=1$.
\end{proof}

\begin{lem}\label{lem:zeta_1(2^n)_gen}
 We have
 \[ 
  \sum_{n\ge0}(\pm2)^{-n}\zeta_1(\{2\}^n)u^{2n+1}=-2F_{\mp}G_{\pm}.
 \]
\end{lem}

\begin{proof}
 Lemma~\ref{lem:zeta_1(2^n)} implies that
 \begin{align*}
  &\sum_{n\ge0}(\pm2)^{-n}\zeta_1(\{2\}^n)u^{2n+1}\\
  &=-2\sum_{n\ge0}(\pm2)^{-n}\sum_{i=0}^{n}(-1)^i\zeta(2i+1)\zeta(\{2\}^{n-i})u^{2n+1}\\
  &=-2\biggl(\sum_{n_1\ge0}(\mp2)^{-n_1}\zeta(2n_1+1)u^{2n_1+1}\biggr)\biggl(\sum_{n_2\ge0}(\pm2)^{-n_2}\zeta(\{2\}^{n_2})u^{2n_2}\biggr)\\
  &=-2F_{\mp}G_{\pm}.\qedhere
 \end{align*}
\end{proof}

\begin{lem}\label{lem:zeta_1(1,3,1)_gen}
 We have
 \[
  \sum_{n\ge0}(-1)^n \zeta_1(\{1,3\}^n,1)u^{4n+2}=\frac{G_+^2-G_-^2}{2G_+G_-}-(F_++F_-)^2G_+G_-.
 \]
\end{lem}

\begin{proof}
 Propositions \ref{zmZ} and \ref{duality} and Lemma~\ref{wordB} show that
 \begin{align*}
  &\sum_{i=0}^{n}(-1)^i\zeta_1(\{2\}^i)\zeta_1(\{2\}^{n-i})\\
  &=\sum_{i=0}^{n}(-1)^iZ^{\sh}(x(yx)^i)Z^{\sh}(x(yx)^{n-i})\\
  &=\sum_{i=0}^{n}(-1)^iZ^{\sh}(x(yx)^i)Z^{\sh}((yx)^{n-i}y)\\
  &=
  \begin{cases}
   (-1)^{n/2}2^n(Z^{\sh}(x(y^2x^2)^{n/2}y)+Z^{\sh}((yx^2y)^{n/2}yx))&\text{if $n$ is even,}\\
   (-1)^{(n-1)/2}2^n(-Z^{\sh}(xy^2(x^2y^2)^{(n-1)/2}x)+Z^{\sh}(y(x^2y^2)^{(n-1)/2}x^2y))&\text{if $n$ is odd}
  \end{cases}\\
  &=
  \begin{cases}
   (-1)^{n/2}2^n(-\zeta_1(\{1,3\}^{n/2},1)+\zeta(\{3,1\}^{n/2},2))&\text{if $n$ is even,}\\
   0&\text{if $n$ is odd.}
  \end{cases}
 \end{align*}
 It follows that
 \begin{align*}
  &\sum_{n\ge0}(-1)^n(-\zeta_1(\{1,3\}^n,1)+\zeta(\{3,1\}^n,2))u^{4n+2}\\
  &=\sum_{\substack{n\ge0\\n:\mathrm{even}}}(-1)^{n/2}(-\zeta_1(\{1,3\}^{n/2},1)+\zeta(\{3,1\}^{n/2},2))u^{2n+2}\\
  &=\sum_{n\ge0}2^{-n}\sum_{i=0}^{n}(-1)^i\zeta_1(\{2\}^i)\zeta_1(\{2\}^{n-i})u^{2n+2}\\
  &=\biggl(\sum_{n_1\ge0}(-2)^{-n_1}\zeta_1(\{2\}^{n_1})u^{2n_1+1}\biggr)\biggl(\sum_{n_2\ge0}2^{-n_2}\zeta_1(\{2\}^{n_2})u^{2n_2+1}\biggr)\\
  &=(-2F_+G_-)(-2F_-G_+)\qquad(\text{Lemma~\ref{lem:zeta_1(2^n)_gen}})\\
  &=4F_+F_-G_+G_-.
 \end{align*}
 Using Lemma~\ref{lem:zeta(3,1,2)_gen}, we infer that
 \begin{align*}
  &\sum_{n\ge0}(-1)^n \zeta_1(\{1,3\}^n,1)u^{4n+2}\\
  &=\sum_{n\ge0}(-1)^n\zeta(\{3,1\}^n,2)u^{4n+2}-4F_+F_-G_+G_-\\
  &=\frac{G_+^2-G_-^2}{2G_+G_-}-(F_+-F_-)^2G_+G_--4F_+F_-G_+G_-\\
  &=\frac{G_+^2-G_-^2}{2G_+G_-}-(F_++F_-)^2G_+G_-.\qedhere
 \end{align*}
\end{proof}

\begin{lem}\label{lem:Z(xyxy)_gen}
 We have
 \[
  \sum_{n\ge0}(\pm2)^{-n}Z^{\sh}((xy)^n)u^{2n}=G_{\pm}^{-1}\pm2F_{\mp}^2G_{\pm}.
 \]
\end{lem}

\begin{proof}
 We have
 \begin{align*}
  \delta_{n,0}
  &=1\sh(xy)^n-y\sh(xy)^{n-1}x+yx\sh(xy)^{n-1}-\cdots-(yx)^{n-1}y\sh x+(yx)^n\sh1\\
  &=\sum_{i=0}^{n}(yx)^i\sh(xy)^{n-i}-\sum_{i=1}^{n}(yx)^{i-1}y\sh(xy)^{n-i}x
 \end{align*}
 for every nonnegative integer $n$.
 It follows that
 \begin{align*}
  1&=\sum_{n\ge0}(\pm2)^{-n}\delta_{n,0}u^{2n}\\
  &=\sum_{n\ge0}(\pm2)^{-n}\biggl(\sum_{i=0}^{n}Z^{\sh}((yx)^i)Z^{\sh}((xy)^{n-i})-\sum_{i=1}^{n}Z^{\sh}((yx)^{i-1}y)Z^{\sh}((xy)^{n-i}x)\biggr)u^{2n}\\
  &=\sum_{n\ge0}(\pm2)^{-n}\biggl(\sum_{i=0}^{n}\zeta(\{2\}^i)Z^{\sh}((xy)^{n-i})-\sum_{i=1}^{n}\zeta_1(\{2\}^{i-1})\zeta_1(\{2\}^{n-i})\biggr)u^{2n}\\
  &\hspace*{250pt}(\text{Propositions \ref{zmZ} and \ref{duality}})\\
  &=\biggl(\sum_{n_1\ge0}(\pm2)^{-n_1}\zeta(\{2\}^{n_1})u^{2n_1}\biggr)\biggl(\sum_{n_2\ge0}(\pm2)^{-n_2}Z^{\sh}((xy)^{n_2})u^{2n_2}\biggr)\\
  &\hphantom{{}={}}\mp2^{-1}\biggl(\sum_{n_1\ge0}(\pm2)^{-n_1}\zeta_1(\{2\}^{n_1})u^{2n_1+1}\biggr)\biggl(\sum_{n_2\ge0}(\pm2)^{-n_2}\zeta_1(\{2\}^{n_2})u^{2n_2+1}\biggr)\\
  &=G_{\pm}\sum_{n\ge0}(\pm2)^{-n}Z^{\sh}((xy)^n)u^{2n}\mp2^{-1}(-2F_{\mp}G_{\pm})^2\qquad(\text{Lemma~\ref{lem:zeta_1(2^n)_gen}})\\
  &=G_{\pm}\sum_{n\ge0}(\pm2)^{-n}Z^{\sh}((xy)^n)u^{2n}\mp2F_{\mp}^2G_{\pm}^2,
 \end{align*}
 which completes the proof.
\end{proof}

\begin{lem}\label{lem:zeta_2(1,3,1,3)_gen}
 We have
 \[
  \sum_{n\ge0}(-1)^n\zeta_2(\{1,3\}^n)u^{4n+2}=2F_+F_-G_+G_-.
 \]
\end{lem}

\begin{proof}
 We have
 \begin{align*}
  &\sum_{n\ge0}(-1)^n\zeta_2(\{1,3\}^n)u^{4n+2}\\
  &=\sum_{n\ge0}(-1)^nZ^{\sh}(x^2(y^2x^2)^n)u^{4n+2}\qquad(\text{Proposition~\ref{zmZ}})\\
  &=\sum_{\substack{n\ge0\\n:\mathrm{even}}}(-1)^{n/2}Z^{\sh}(x^2(y^2x^2)^{n/2})u^{2n+2}\\
  &=\sum_{n\ge0}2^{-n-1}\sum_{i=0}^{n}(-1)^iZ^{\sh}(x(yx)^i)Z^{\sh}(x(yx)^{n-i})u^{2n+2}\qquad(\text{Lemma~\ref{word2}})\\
  &=\sum_{n\ge0}2^{-n-1}\sum_{i=0}^{n}(-1)^i\zeta_1(\{2\}^i)\zeta_1(\{2\}^{n-i})u^{2n+2}\qquad(\text{Proposition~\ref{zmZ}})\\
  &=2^{-1}\biggl(\sum_{n_1\ge0}(-2)^{-n_1}\zeta_1(\{2\}^{n_1})u^{2n_1+1}\biggr)\biggl(\sum_{n_2\ge0}2^{-n_2}\zeta_1(\{2\}^{n_2})u^{2n_2+1}\biggr)\\
  &=2^{-1}(-2F_+G_-)(-2F_-G_+)\qquad(\text{Lemma~\ref{lem:zeta_1(2^n)_gen}})\\
  &=2F_+F_-G_+G_-.\qedhere
 \end{align*}
\end{proof}

\begin{lem}\label{lem:zeta_2(1,3,1)_gen}
 We have
 \[
  \sum_{n\ge0}(-1)^n\zeta_2(\{1,3\}^n,1)u^{4n+3}=\frac{F_+G_-^2-F_-G_+^2}{G_+G_-}+2F_+F_-(F_++F_-)G_+G_-.
 \]
\end{lem}

\begin{proof}
 We have
 \begin{align*}
  &\sum_{n\ge0}(-1)^n\zeta_2(\{1,3\}^n,1)u^{4n+3}\\
  &=\sum_{n\ge0}(-1)^nZ^{\sh}(x^2(y^2x^2)^ny)u^{4n+3}\qquad(\text{Proposition~\ref{zmZ}})\\
  &=\sum_{n\ge0}2^{-2n-1}\sum_{i=0}^{2n+1}(-1)^iZ^{\sh}(x(yx)^i)Z^{\sh}((xy)^{2n-i+1})u^{4n+3}\qquad(\text{Lemma~\ref{word1}})\\
  &=-\sum_{n\ge0}2^{-2n-1}\sum_{i=0}^{2n+1}(-1)^i\zeta_1(\{2\}^i)Z^{\sh}((xy)^{2n-i+1})u^{4n+3}\qquad(\text{Proposition~\ref{zmZ}})\\
  &=-\sum_{\substack{n_1,n_2\ge0\\n_1+n_2:\mathrm{odd}}}2^{-n_1-n_2}(-1)^{-n_1}\zeta_1(\{2\}^{n_1})Z^{\sh}((xy)^{n_2})u^{2(n_1+n_2)+1}\\
  &=-2^{-1}\sum_{n_1,n_2\ge0}(1-(-1)^{n_1+n_2})2^{-n_1-n_2}(-1)^{-n_1}\zeta_1(\{2\}^{n_1})Z^{\sh}((xy)^{n_2})u^{2(n_1+n_2)+1}\\
  &=-2^{-1}\biggl(\biggl(\sum_{n_1\ge0}(-2)^{-n_1}\zeta_1(\{2\}^{n_1})u^{2n_1+1}\biggr)\biggl(\sum_{n_2\ge0}2^{-n_2}Z^{\sh}((xy)^{n_2})u^{2n_2}\biggr)\\
  &\hphantom{=2^{-1}\biggl(}-\biggl(\sum_{n_1\ge0}2^{-n_1}\zeta_1(\{2\}^{n_1})u^{2n_1+1}\biggr)\biggl(\sum_{n_2\ge0}(-2)^{-n_2}Z^{\sh}((xy)^{n_2})u^{2n_2}\biggr)\biggr)\\
  &=-2^{-1}(-2F_+G_-(G_+^{-1}+2F_-^2G_+)+2F_-G_+(G_-^{-1}-2F_+^2G_-))\qquad(\text{Lemmas~\ref{lem:zeta_1(2^n)_gen} and \ref{lem:Z(xyxy)_gen}})\\
  &=\frac{F_+G_-^2-F_-G_+^2}{G_+G_-}+2F_+F_-(F_++F_-)G_+G_-.\qedhere
 \end{align*}
\end{proof}

\subsection{Proofs of Theorems \ref{main1} and \ref{main2}}
\begin{proof}[Proof of Theorem~\ref{main1}]
 Since
 \begin{align*}
  &\zeta_{\mathcal{S}_3}(\{3,1\}^n)\\
  &=\sum_{i=0}^{n}\zeta(\{3,1\}^i)\sum_{m=0}^{2}\zeta_m(\{1,3\}^{n-i})t^m-\sum_{i=1}^{n}\zeta(\{3,1\}^{i-1},3)\sum_{m=0}^{2}\zeta_m(\{1,3\}^{n-i},1)t^m
 \end{align*}
 for every nonnegative integer $n$, we have
 \begin{align*}
  &\sum_{n\ge0}(-1)^n\zeta_{\mathcal{S}_3}(\{3,1\}^n)u^{4n}\\
  &=\sum_{n\ge0}(-1)^n\sum_{i=0}^{n}\zeta(\{3,1\}^i)\sum_{m=0}^{2}\zeta_m(\{1,3\}^{n-i})t^mu^{4n}\\
  &\hphantom{{}={}}-\sum_{n\ge0}(-1)^n\sum_{i=1}^{n}\zeta(\{3,1\}^{i-1},3)\sum_{m=0}^{2}\zeta_m(\{1,3\}^{n-i},1)t^mu^{4n}\\
  &=\biggl(\sum_{n_1\ge0}(-1)^{n_1}\zeta(\{3,1\}^{n_1})u^{4n_1}\biggr)\biggl(\sum_{m=0}^{2}\sum_{n_2\ge0}(-1)^{n_2}\zeta_m(\{1,3\}^{n_2})u^{4n_2}t^m\biggr)\\
  &\hphantom{{}={}}+\biggl(\sum_{n_1\ge0}(-1)^{n_1}\zeta(\{3,1\}^{n_1},3)u^{4n_1+3}\biggr)\biggl(\sum_{m=0}^{2}\sum_{n_2\ge0}(-1)^{n_2}\zeta_m(\{1,3\}^{n_2},1)u^{4n_2+1}t^m\biggr)\\
  &=\biggl(\frac{G_+^2+G_-^2}{2G_+G_-}-(F_+^2-F_-^2)G_+G_-\biggr)\biggl(G_+G_--(F_++F_-)G_+G_-\frac{t}{u}+2F_+F_-G_+G_-\frac{t^2}{u^2}\biggr)\\
  &\hphantom{{}={}}+(F_+-F_-)G_+G_-\biggl((F_++F_-)G_+G_-+\biggl(\frac{G_+^2-G_-^2}{2G_+G_-}-(F_++F_-)^2G_+G_-\biggr)\frac{t}{u}\\
  &\hphantom{{}={}+(F_+-F_-)G_+G_-\biggl(}+\biggl(\frac{F_+G_-^2-F_-G_+^2}{G_+G_-}+2F_+F_-(F_++F_-)G_+G_-\biggr)\frac{t^2}{u^2}\biggr)\\
  &=\frac{G_+^2+G_-^2}{2}-(F_+G_-^2+F_-G_+^2)\frac{t}{u}+(F_+^2G_-^2+F_-^2G_+^2)\frac{t^2}{u^2}.
 \end{align*}
 by Lemmas \ref{lem:zeta(4^n)_gen}, \ref{lem:zeta(3,1,3)_gen}, \ref{lem:zeta(1,3,1)_gen}, \ref{lem:zeta(3,1,3,1)_gen}, \ref{lem:zeta_1(1,3,1,3)_gen}, \ref{lem:zeta_1(1,3,1)_gen}, \ref{lem:zeta_2(1,3,1,3)_gen}, and \ref{lem:zeta_2(1,3,1)_gen}.

 Now since Lemma \ref{lem:zeta(4^n)_gen} implies that
 \[
  \frac{G_+^2+G_-^2}{2}=\sum_{\substack{n\ge0}}\frac{2^{2n+1}\pi^{4n}}{(4n+2)!}u^{4n},
 \]
 that
 \begin{align*}
  -(F_+G_-^2+F_-G_+^2)\frac{t}{u}
  &=-\biggl(\sum_{n_1\ge0}2^{-n_1}\zeta(2n_1+1)u^{2n_1+1}\biggr)\biggl(\sum_{n_0\ge0}\frac{(-1)^{n_0}2^{n_0+1}\pi^{2n_0}}{(2n_0+2)!}u^{2n_0}\biggr)\frac{t}{u}\\
  &\hphantom{{}={}}-\biggl(\sum_{n_1\ge0}(-2)^{-n_1}\zeta(2n_1+1)u^{2n_1+1}\biggr)\biggl(\sum_{n_0\ge0}\frac{2^{n_0+1}\pi^{2n_0}}{(2n_0+2)!}u^{2n_0}\biggr)\frac{t}{u}\\
  &=-\sum_{n_0,n_1\ge0}\frac{((-1)^{n_0}+(-1)^{n_1})2^{n_0-n_1+1}}{(2n_0+2)!}\pi^{2n_0}\zeta(2n_1+1)u^{2n_0+2n_1}t\\
  &=-\sum_{\substack{n_0,n_1\ge0\\n_0+n_1:\mathrm{even}}}\frac{(-1)^{n_0}2^{n_0-n_1+2}}{(2n_0+2)!}\pi^{2n_0}\zeta(2n_1+1)u^{2n_0+2n_1}t,
 \end{align*}
 and that
 \begin{align*}
  &(F_+^2G_-^2+F_-^2G_+^2)\frac{t^2}{u^2}\\
  &=\biggl(\sum_{n_1\ge0}2^{-n_1}\zeta(2n_1+1)u^{2n_1+1}\biggr)\biggl(\sum_{n_2\ge0}2^{-n_2}\zeta(2n_2+1)u^{2n_2+1}\biggr)\biggl(\sum_{n_0\ge0}\frac{(-1)^{n_0}2^{n_0+1}\pi^{2n_0}}{(2n_0+2)!}u^{2n_0}\biggr)\frac{t^2}{u^2}\\
  &\hphantom{{}={}}+\biggl(\sum_{n_1\ge0}(-2)^{-n_1}\zeta(2n_1+1)u^{2n_1+1}\biggr)\biggl(\sum_{n_2\ge0}(-2)^{-n_2}\zeta(2n_2+1)u^{2n_2+1}\biggr)\biggl(\sum_{n_0\ge0}\frac{2^{n_0+1}\pi^{2n_0}}{(2n_0+2)!}u^{2n_0}\biggr)\frac{t^2}{u^2}\\
  &=\sum_{n_0,n_1,n_2\ge0}\frac{((-1)^{n_0}+(-1)^{n_1+n_2})2^{n_0-n_1-n_2+1}}{(2n_0+2)!}\pi^{2n_0}\zeta(2n_1+1)\zeta(2n_2+1)u^{2(n_0+n_1+n_2)}t^2\\
  &=\sum_{\substack{n_0,n_1,n_2\ge0\\n_0+n_1+n_2:\mathrm{even}}}\frac{(-1)^{n_0}2^{n_0-n_1-n_2+2}}{(2n_0+2)!}\pi^{2n_0}\zeta(2n_1+1)\zeta(2n_2+1)u^{2(n_0+n_1+n_2)},
 \end{align*}
 we have
 \begin{align*}
  (-1)^n\zeta_{\mathcal{S}_3}(\{3,1\}^n)
  &=\frac{2^{2n+1}}{(4n+2)!}\pi^{4n}-\sum_{\substack{n_0,n_1\ge0\\n_0+n_1=2n}}\frac{(-1)^{n_0}2^{n_0-n_1+2}}{(2n_0+2)!}\pi^{2n_0}\zeta(2n_1+1)t\\
  &\hphantom{{}={}}+\sum_{\substack{n_0,n_1,n_2\ge0\\n_0+n_1+n_2=2n}}\frac{(-1)^{n_0}2^{n_0-n_1-n_2+2}}{(2n_0+2)!}\pi^{2n_0}\zeta(2n_1+1)\zeta(2n_2+1)t^2,
 \end{align*}
 from which the theorem follows.
\end{proof}

\begin{proof}[Proof of Theorem~\ref{main2}]
 Since
 \begin{align*}
  &\zeta_{\mathcal{S}_3}(\{1,3\}^n,1)\\
  &=-\sum_{i=0}^{n}\zeta(\{1,3\}^i)\sum_{m=0}^{2}\zeta_m(\{1,3\}^{n-i},1)t^m+\sum_{i=0}^{n}\zeta(\{1,3\}^i,1)\sum_{m=0}^{2}\zeta_m(\{1,3\}^{n-i})t^m
 \end{align*}
 for every nonnegative integer $n$, we have
 \begin{align*}
  &\sum_{n\ge0}(-1)^n\zeta_{\mathcal{S}_3}(\{1,3\}^n,1)u^{4n+1}\\
  &=-\sum_{n\ge0}(-1)^n\sum_{i=0}^{n}\zeta(\{1,3\}^i)\sum_{m=0}^{2}\zeta_m(\{1,3\}^{n-i},1)t^mu^{4n+1}\\
  &\hphantom{{}={}}+\sum_{n\ge0}(-1)^n\sum_{i=0}^{n}\zeta(\{1,3\}^i,1)\sum_{m=0}^{2}\zeta_m(\{1,3\}^{n-i})t^mu^{4n+1}\\
  &=-\biggl(\sum_{n_1\ge0}(-1)^{n_1}\zeta(\{1,3\}^{n_1})u^{4n_1}\biggr)\biggl(\sum_{m=0}^{2}\sum_{n_2\ge0}(-1)^{n_2}\zeta_m(\{1,3\}^{n_2},1)u^{4n_2+1}t^m\biggr)\\
  &\hphantom{{}={}}+\biggl(\sum_{n_1\ge0}(-1)^{n_1}\zeta(\{1,3\}^{n_1},1)u^{4n_1+1}\biggr)\biggl(\sum_{m=0}^{2}\sum_{n_2\ge0}(-1)^{n_2}\zeta_m(\{1,3\}^{n_2})u^{4n_2}t^m\biggr)\\
  &=-G_+G_-\biggl((F_++F_-)G_+G_-+\biggl(\frac{G_+^2-G_-^2}{2G_+G_-}-(F_++F_-)^2G_+G_-\biggr)\frac{t}{u}\\
  &\hphantom{=-G_+G_-\biggl(}+\biggl(\frac{F_+G_-^2-F_-G_+^2}{G_+G_-}+2F_+F_-(F_++F_-)G_+G_-\biggr)\frac{t^2}{u^2}\biggr)\\
  &\hphantom{{}={}}+(F_++F_-)G_+G_-\biggl(G_+G_--(F_++F_-)G_+G_-\frac{t}{u}+2F_+F_-G_+G_-\frac{t^2}{u^2}\biggr)\\
  &=-\frac{G_+^2-G_-^2}{2}\cdot\frac{t}{u}+(-F_+G_-^2+F_-G_+^2)\frac{t^2}{u^2}
 \end{align*}
 by Lemmas \ref{lem:zeta(4^n)_gen}, \ref{lem:zeta(1,3,1)_gen}, \ref{lem:zeta_1(1,3,1,3)_gen}, \ref{lem:zeta_1(1,3,1)_gen}, \ref{lem:zeta_2(1,3,1,3)_gen}, and \ref{lem:zeta_2(1,3,1)_gen}.

 Now since Lemma~\ref{lem:zeta(4^n)_gen} implies that
 \[
  -\frac{G_+^2-G_-^2}{2}\cdot\frac{t}{u}=-\sum_{\substack{n\ge0\\n:\mathrm{odd}}}\frac{2^{n+1}\pi^{2n}}{(2n+2)!}u^{2n-1}t=-\sum_{n\ge0}\frac{2^{2n+2}\pi^{4n+2}}{(4n+4)!}u^{4n+1}t
 \]
 and that
 \begin{align*}
  &(-F_+G_-^2+F_-G_+^2)\frac{t^2}{u^2}\\
  &=-\biggl(\sum_{n_1\ge0}2^{-n_1}\zeta(2n_1+1)u^{2n_1+1}\biggr)\biggl(\sum_{n_0\ge0}\frac{(-1)^{n_0}2^{n_0+1}\pi^{2n_0}}{(2n_0+2)!}u^{2n_0}\biggr)\frac{t^2}{u^2}\\
  &\hphantom{{}={}}+\biggl(\sum_{n_1\ge0}(-2)^{-n_1}\zeta(2n_1+1)u^{2n_1+1}\biggr)\biggl(\sum_{n_0\ge0}\frac{2^{n_0+1}\pi^{2n_0}}{(2n_0+2)!}u^{2n_0}\biggr)\frac{t^2}{u^2}\\
  &=\sum_{n_0,n_1\ge0}\frac{(-(-1)^{n_0}+(-1)^{n_1})2^{n_0-n_1+1}}{(2n_0+2)!}\pi^{2n_0}\zeta(2n_1+1)u^{2(n_0+n_1)-1}t^2\\
  &=\sum_{\substack{n_0,n_1\ge0\\n_0+n_1:\mathrm{odd}}}\frac{(-1)^{n_1}2^{n_0-n_1+2}}{(2n_0+2)!}\pi^{2n_0}\zeta(2n_1+1)u^{2(n_0+n_1)-1}t^2,
 \end{align*}
 we have
 \[
  (-1)^n\zeta_{\mathcal{S}_3}(\{1,3\}^n,1)
  =-\frac{2^{2n+2}\pi^{4n+2}}{(4n+4)!}t
   +\sum_{\substack{n_0,n_1\ge0\\n_0+n_1=2n+1}}\frac{(-1)^{n_1}2^{n_0-n_1+2}}{(2n_0+2)!}\pi^{2n_0}\zeta(2n_1+1)t^2
 \]
 from which the theorem follows.
\end{proof}

\subsection{Proof of Theorem~\ref{main0}}
We define $\mathbb{Q}$-linear maps $I_0,I_1\colon\mathcal{I}\to\mathcal{I}$ by
\begin{align*}
 I_0(k_1,\dots,k_r)&=\sum_{i=0}^{r}(-1)^{k_{i+1}+\dots+k_r}(k_1,\dots,k_i)*(k_r,\dots,k_{i+1}),\\
 I_1(k_1,\dots,k_r)&=\sum_{i=0}^{r}(-1)^{k_{i+1}+\dots+k_r}(k_1,\dots,k_i)*\sigma(k_r,\dots,k_{i+1})
\end{align*}
for all indices $(k_1,\dots,k_r)$, where $\sigma\colon\mathcal{I}\to\mathcal{I}$ is the $\mathbb{Q}$-linear map defined by
\begin{align*}
 \sigma(k_1,\dots,k_r)
 &=\sum_{\substack{l_1,\dots,l_r\ge0\\l_1+\dots+l_r=1}}(k_1+l_1,\dots,k_r+l_r)\prod_{i=1}^{r}\binom{k_i+l_i-1}{l_i}\\
 &=\sum_{i=1}^{r}k_i(k_1,\dots,k_{i-1},k_i+1,k_{i+1},\dots,k_r).
\end{align*}
Observe that
\[
 \zeta_{\mathcal{S}_2}^*(k_1,\dots,k_r)=\zeta^*(I_0(k_1,\dots,k_r))+\zeta^*(I_1(k_1,\dots,k_r))t
\]
and that
\[
\sigma(\boldsymbol{k}*\boldsymbol{l})=\sigma(\boldsymbol{k})*\boldsymbol{l}+\boldsymbol{k}*\sigma(\boldsymbol{l})
\]
for every $\boldsymbol{k},\boldsymbol{l}\in\mathcal{I}$.

\begin{lem}\label{lem:I_0(a,b,a,b)}
 If $a$ and $b$ are odd positive integers, then we have
 \[
  I_0(\{a,b\}^n)=(-1)^n(\{a+b\}^n)
 \]
 for every nonnegative integer $n$.
\end{lem}

\begin{proof}
 We proceed by induction on $n$.
 The assertion is obvious for $n=0$; suppose that it is true for $n$.
 Write $k_i=a$ for odd $i$ and $k_i=b$ for even $i$.
 Then we have
 \begin{align*}
  &I_0(\{a,b\}^{n+1})\\
  &=I_0(k_1,\dots,k_{2n+2})\\
  &=\sum_{i=0}^{2n+2}(-1)^{k_{i+1}+\dots+k_{2n+2}}(k_1,\dots,k_i)*(k_{2n+2},\dots,k_{i+1})\\
  &=\sum_{i=0}^{2n+2}(-1)^i(k_1,\dots,k_i)*(k_{2n+2},\dots,k_{i+1})\\
  &=\sum_{i=1}^{2n+1}(-1)^i(((k_1,\dots,k_{i-1})*(k_{2n+2},\dots,k_{i+1}),k_i)+((k_1,\dots,k_i)*(k_{2n+2},\dots,k_{i+2}),k_{i+1}))\\
  &\hphantom{{}={}}+\sum_{i=1}^{2n+1}(-1)^i((k_1,\dots,k_{i-1})*(k_{2n+2},\dots,k_{i+2}),k_i+k_{i+1})\\
  &\hphantom{{}={}}+(k_{2n+2},\dots,k_1)+(k_1,\dots,k_{2n+2})\\
  &=\sum_{i=1}^{2n+1}(-1)^i((k_1,\dots,k_{i-1})*(k_{2n+2},\dots,k_{i+2}),k_i+k_{i+1})\\
  &=-\biggl(\sum_{i=0}^{2n}(-1)^i(k_1,\dots,k_i)*(k_{2n},\dots,k_{i+1}),a+b\biggr)\\
  &=-(I_0(k_1,\dots,k_{2n}),a+b)\\
  &=-(I_0(\{a,b\}^n),a+b)\\
  &=-(-1)^n(\{a+b\}^n,a+b)\qquad(\text{induction hypothesis})\\
  &=(-1)^{n+1}(\{a+b\}^{n+1}).\qedhere
 \end{align*}
\end{proof}

\begin{lem}\label{lem:I_1_reversal_sum}
 If $k_1,\dots,k_r$ are positive integers with $k_1+\dots+k_r$ even, then we have
 \[
  I_1(k_1,\dots,k_r)+I_1(k_r,\dots,k_1)=\sigma(I_0(k_1,\dots,k_r)).
 \]
\end{lem}

\begin{proof}
 We have
 \begin{align*}
  &I_1(k_1,\dots,k_r)+I_1(k_r,\dots,k_1)\\
  &=\sum_{i=0}^{r}(-1)^{k_{i+1}+\dots+k_r}(k_1,\dots,k_i)*\sigma(k_r,\dots,k_{i+1})\\
  &\hphantom{{}={}}+\sum_{i=0}^{r}(-1)^{k_1+\dots+k_i}(k_r,\dots,k_{i+1})*\sigma(k_1,\dots,k_i)\\
  &=\sum_{i=0}^{r}(-1)^{k_{i+1}+\dots+k_r}((k_1,\dots,k_i)*\sigma(k_r,\dots,k_{i+1})+\sigma(k_1,\dots,k_i)*(k_r,\dots,k_{i+1}))\\
  &=\sum_{i=0}^{r}(-1)^{k_{i+1}+\dots+k_r}\sigma((k_1,\dots,k_i)*(k_r,\dots,k_{i+1}))\\
  &=\sigma(I_0(k_1,\dots,k_r)).\qedhere
 \end{align*}
\end{proof}

\begin{lem}\label{lem:I_1(a,b,a,b)+I_1(b,a,b,a)}
 We have
 \[
  I_1(\{1,3\}^n)+I_1(\{3,1\}^n)=4(-1)^n\sum_{i=0}^{n-1}(-1)^i(4i+5)*(\{4\}^{n-i-1})
 \]
 for every positive integer $n$.
\end{lem}

\begin{proof}
 We have
 \begin{align*}
  I_1(\{1,3\}^n)+I_1(\{3,1\}^n)
  &=\sigma(I_0(\{1,3\}^n))\qquad(\text{Lemma~\ref{lem:I_1_reversal_sum}})\\
  &=(-1)^n\sigma(\{4\}^n)\qquad(\text{Lemma~\ref{lem:I_0(a,b,a,b)}})\\
  &=\binom{4+1-1}{1}(-1)^n((\{4\}^{n-1})\sha(5))\\
  &=4(-1)^n\sum_{i=0}^{n-1}(-1)^i(4i+5)*(\{4\}^{n-i-1})\qquad(\text{Lemma~\ref{lem:sha_alternating_sum}}).\qedhere
 \end{align*}
\end{proof}

\begin{proof}[Proof of Theorem~\ref{main0}]
 Lemma~\ref{lem:I_0(a,b,a,b)} shows that
 \[
  \zeta_{\mathcal{S}_1}(\{1,3\}^n)
  =\zeta^*(I_0(\{1,3\}^n))
  =(-1)^n\zeta(\{4\}^n)
  =\frac{2(-4)^n}{(4n+2)!}\pi^{4n}
 \]
 for every nonnegative integer $n$.

 We now compute $\zeta(I_1(\{1,3\}^n))$, the coefficient of $t$ in $\zeta_{\mathcal{S}_2}(\{1,3\}^n)$, for nonnegative integers $n$.
 Since we obviously have $\zeta(I_1(\{1,3\}^0))=0$, we assume that $n\ge1$.
 Lemma~\ref{lem:I_1(a,b,a,b)+I_1(b,a,b,a)} shows that
 \begin{align*}
  &\zeta(I_1(\{1,3\}^n))+\zeta(I_1(\{3,1\}^n))\\
  &=4(-1)^n\sum_{i=0}^{n-1}(-1)^i\zeta(4i+5)\zeta(\{4\}^{n-i-1})\\
  &=-4\sum_{\substack{n_0,n_1\ge0\\n_0+n_1=n}}(-1)^{n_0}\zeta(\{4\}^{n_0})\zeta(4n_1+1)\\
  &=-4\sum_{\substack{n_0,n_1\ge0\\n_0+n_1=n}}\frac{2(-4)^{n_0}}{(4n_0+2)!}\pi^{4n_0}\zeta(4n_1+1).
 \end{align*}
 Since Theorem~\ref{main1} shows that
 \begin{align*}
  &\zeta(I_1(\{3,1\}^n))\\
  &=(-1)^{n+1}\sum_{\substack{n_0,n_1\ge0\\n_0+n_1=2n}}\frac{(-1)^{n_0}2^{n_0-n_1+2}}{(2n_0+2)!}\pi^{2n_0}\zeta(2n_1+1)\\
  &=-(-1)^n\sum_{\substack{n_0,n_1\ge0\\n_0+n_1=2n\\n_0,n_1:\mathrm{even}}}\frac{2^{n_0-n_1+2}}{(2n_0+2)!}\pi^{2n_0}\zeta(2n_1+1)+(-1)^{n}\sum_{\substack{n_0,n_1\ge0\\n_0+n_1=2n\\n_0,n_1:\mathrm{odd}}}\frac{2^{n_0-n_1+2}}{(2n_0+2)!}\pi^{2n_0}\zeta(2n_1+1)\\
  &=-(-1)^n\sum_{\substack{n_0,n_1\ge0\\n_0+n_1=n}}\frac{2^{2n_0-2n_1+2}}{(4n_0+2)!}\pi^{4n_0}\zeta(4n_1+1)+(-1)^n\sum_{\substack{n_0,n_1\ge0\\n_0+n_1=2n\\n_0,n_1:\mathrm{odd}}}\frac{2^{n_0-n_1+2}}{(2n_0+2)!}\pi^{2n_0}\zeta(2n_1+1),
 \end{align*}
 we obtain
 \begin{align*}
  &\zeta(I_1(\{1,3\}^n))\\
  &=-4\sum_{\substack{n_0,n_1\ge0\\n_0+n_1=n}}\frac{2(-4)^{n_0}}{(4n_0+2)!}\pi^{4n_0}\zeta(4n_1+1)\\
  &\hphantom{{}={}}+(-1)^n\sum_{\substack{n_0,n_1\ge0\\n_0+n_1=n}}\frac{2^{2n_0-2n_1+2}}{(4n_0+2)!}\pi^{4n_0}\zeta(4n_1+1)-(-1)^n\sum_{\substack{n_0,n_1\ge0\\n_0+n_1=2n\\n_0,n_1:\mathrm{odd}}}\frac{2^{n_0-n_1+2}}{(2n_0+2)!}\pi^{2n_0}\zeta(2n_1+1)\\
  &=\sum_{\substack{n_0,n_1\ge0\\n_0+n_1=n}}\frac{(-4)^{n_0+1}(2-(-4)^{-n_1})}{(4n_0+2)!}\pi^{4n_0}\zeta(4n_1+1)-(-1)^n\sum_{\substack{n_0,n_1\ge0\\n_0+n_1=2n\\n_0,n_1:\mathrm{odd}}}\frac{2^{n_0-n_1+2}}{(2n_0+2)!}\pi^{2n_0}\zeta(2n_1+1),
 \end{align*}
 as required.
\end{proof}

\section*{Acknowledgements}
This work was supported by JSPS KAKENHI Grant Numbers JP18K03243 and JP18K13392 and by Grant for Basic Science Research Projects from The Sumitomo Foundation.

\end{document}